\documentclass[12pt]{article}
\RequirePackage{amsthm,amsmath}
\RequirePackage{natbib}
\usepackage{mathrsfs}
\usepackage{amssymb}
\usepackage{amsbsy}
\usepackage{epsfig}
\usepackage{enumerate}
\usepackage{bm}
\usepackage{color}
\usepackage{booktabs}
\usepackage{hyperref} 
\usepackage{fullpage}
\usepackage{setspace}
\usepackage{graphicx}
\usepackage{caption}
\usepackage{subcaption}
\usepackage{multirow}
\usepackage{color}
\usepackage{verbatim}
\usepackage{bm}
\usepackage{mathrsfs}
\usepackage{bbm}
\usepackage{graphics}
\theoremstyle{plain}

\newtheorem{theorem}{Theorem}[section]

\newtheorem{corollary}[theorem]{Corollary}

\newtheorem{assumption}{Assumption}

\allowdisplaybreaks

\graphicspath{{./figs/}}
\DeclareGraphicsExtensions{.eps,.ps,.pdf}

\newcommand{\be}{\begin{eqnarray}}
\newcommand{\ee}{\end{eqnarray}}
\newcommand{\ben}{\begin{eqnarray*}}
\newcommand{\een}{\end{eqnarray*}}

\begin{document}

\title{Statistical Inference for Covariate-Adaptive Randomization Procedures}
\author{Wei~Ma,~Yichen~Qin,~Yang~Li,~and~Feifang~Hu\thanks{
Wei Ma is Assistant Professor at Renmin University of China.  
Yichen Qin is Assistant Professor at University of Cincinnati.  
Yang Li is Professor at Renmin University of China.  
Feifang Hu (feifang@gwu.edu) is Professor at George Washington University.  
This work is partially supported by grant DMS-1612970 from the 
National Science Foundation and by grant 11801559, 11371366, 11731011 and 71771211 from the 
National Natural Science Foundation of China.}}
\date{}
\maketitle

\begin{abstract}
Covariate-adaptive randomization procedures are frequently used in comparative studies to 
increase the covariate balance across treatment groups.  
However, because randomization inevitably uses the 
covariate information when forming balanced treatment groups, 
the validity of classical statistical methods after such
randomization is often unclear.
In this article, we derive the theoretical properties of 
statistical methods based on 
general covariate-adaptive randomization under the linear model framework. 
More importantly, we explicitly unveil the relationship between 
covariate-adaptive and inference properties by deriving the asymptotic representations of the corresponding estimators.  
We apply the proposed general theory to various randomization procedures such 
as complete randomization (CR), rerandomization (RR), pairwise 
sequential randomization (PSR), and Atkinson's $D_A$-biased coin 
design ($D_A$-BCD) and compare their performance analytically.
Based on the theoretical results, we then propose a 
new approach to obtain valid and more powerful tests. These results open a door to 
understand and analyze experiments based on covariate-adaptive 
randomization. Simulation studies provide further evidence of the advantages of the proposed 
framework and the theoretical results.
Supplementary materials for this article are available online.

\textbf{Keywords:} Balancing covariates;
Conservative tests;
Power;
Sequential analysis;
Asymptotic normality.
\end{abstract}

\section{Introduction}\label{sec:Introduction}

Randomization is considered the ``gold standard'' to evaluate the treatment 
effect because it mitigates selection bias and provides a foundation for statistical 
inference.
Among the randomization methods, covariate-adaptive randomization (CAR) procedures are frequently used because they use the covariate information to form more 
balanced treatment groups.
However, this feature usually compromises
the validity of classical statistical inference following such
randomization.
In this article, we establish a general theory by which 
properties of statistical inference can be obtained for CAR under widely satisfied conditions.

CAR procedures have been studied extensively.
For categorical covariates, stratified randomization, 
Pocock and Simon's minimization method, and 
its extensions can be used to reduce covariate imbalance at different levels
\citep{Pocock1975,Hu2012} and can handle continuous 
covariates via discretization.
To avoid information loss due to discretization, many randomization methods 
that make direct use of continuous covariates have also been proposed 
\citep{Frane1998,Lin2012,Ma2013}.
Atkinson's $D_A$-biased coin design ($D_A$-BCD) represents a large 
class of methods that accounts for covariates in allocation rules based on certain optimality 
criteria \citep{Atkinson1982,Smith1984,Smith1984b,Antognini2011}.
When all units' covariates are available before the experiment begins, 
we can adopt rerandomization (RR), which repeats the traditional 
randomization process until a satisfactory configuration is achieved 
\citep{Morgan2012}.
Another alternative is pairwise sequential randomization (PSR), recently proposed by 
\citet{Qin2016}, which achieves the optimal covariate 
balance and has greater computational efficiency.
For an overview, please see 
\citet{Hu2014} and \citet{Rosenberger2015}.

Because the aforementioned randomizations inevitably use the covariate 
information to form more balanced treatment groups, 
the subsequent statistical inference is usually affected 
and demonstrates undesirable properties, such as reduced type I errors
\citep{Shao2010, Ma2015}.
This phenomenon of conservativeness is particularly common for a working model that 
includes only a subset of covariates used in randomization, such as the two-sample t-test.
As all of the covariates are used in the randomization to generate more 
balanced assignments, a valid statistical procedure should incorporate all covariates.
Therefore, excluding some covariates from the working model leads to distortion of the sampling distribution of test statistics, which consequently 
causes invalid statistical inference.

It is ideal that the covariates used in randomization should be 
included in the subsequent analysis in the context of clinical 
trials according to regulatory guidelines 
\citep{ICH1998,EMA2015}.
However, unadjusted analyses are still commonly used in practice \citep{Kahan2014,Sverdlov2015}.
For example, to avoid the inclusion of an excessive number of parameters, investigation sites are usually 
omitted in the analysis model for a multicenter clinical trial.
Other practical reasons to avoid incorporating all covariates 
include the simplicity of the test procedure and
robustness to model misspecification \citep{Shao2010}.
Therefore, many working models may suffer from 
the issue of invalid statistical inference.
As covariates are commonly used in comparative studies such 
as biomarker analysis and
crowdsourced-internet experimentation 
\citep{Horton2011,Chandler2013}, understanding the 
impact of CAR on statistical inference is an 
increasingly pressing problem.

The issue over the validity of statistical inference after balancing covariates 
was mainly investigated on the basis of simulations in the early literature, such 
as \citet{Birkett1985,Forsythe1987}.  
More recently, theoretical progress has been made on the inference 
properties for some specific CAR procedures.
\citet{Shao2010} prove that the two-sample t-test is conservative under a 
special stratified randomization.
\citet{Ma2015} study hypotheses testing under a linear 
model for discrete covariate-adaptive randomization, which assumes that the 
overall and marginal imbalances across covariates are bounded in probability.
\cite{Bugni2018} also require stratification to study inference for CAR but allow that the fraction of units within each stratum tends to be normally distributed.
However, their results are limited because many CAR procedures deal 
with continuous covariates directly and do not necessarily satisfy the strong 
balancing assumptions.
In this article, we study inference properties under a general framework and 
demonstrate the impact of CAR on inference.

The main contributions of this article are as follows. 
First, we derive the properties of statistical inference following general 
CAR procedures under the linear model 
framework. Most importantly, we explicitly display the relationship between covariate balance and 
inference by deriving their asymptotic representations. 
This result explains why inference shows differences in behavior with various 
randomizations.
Second, we show that the results have broad applications, as illustrated 
by application in several randomization procedures, such as complete randomization (CR), RR, PSR, and 
$D_A$-BCD.
In addition, it provides a theoretical approach for formal evaluation of 
inference properties and comparison of the pros and cons of different CAR procedures.
Third, we propose a method to obtain valid and powerful tests based on our 
theoretical results. 
The study lays a foundation to understand the impact of covariate balance on statistical inference and illuminates future study in 
this area.
 
The general framework is introduced in Section \ref{sec:Framework}, and the theoretical results are presented in Section \ref{sec:HypoTestGeneral}.
We study four specific randomization procedures in 
terms of their conservativeness in hypothesis testing in 
Section \ref{sec:HypoTestSpecial} and propose a corrected test in Section \ref{sec:correction}.
In Section \ref{sec:Simulation} and \ref{sec:RealData}, numerical studies and a real data example are presented to illustrate the effectiveness of the proposed theory.
The extension to logistic regression models is discussed in Section \ref{sec:Logistic}.
Section \ref{sec:Conclusion} concludes with some remarks and suggested topics for future 
research. The technical proofs are given in the supplementary materials.

\section{General Framework}\label{sec:Framework}

Suppose that $n$ units are to be assigned to two treatment groups using a 
CAR procedure.
Let $T_i$ be the assignment of the $i$-th unit, i.e., $T_i=1$ for treatment 1 
and $T_i=0$ for treatment 2.
Let $\bm{x}_i=(x_{i,1},...,x_{i,p+q})^t$ represent $p+q$ covariates observed 
for the $i$-th unit,
where $x_{i,j}$ are independent and identically distributed as 
$X_j$ for each unit $i=1,...,n$.
A linear regression model is assumed for the outcome $Y_i$ of the 
$i$-th unit,
\begin{align}\label{eq:true_model}
Y_i &= \mu_1T_i + \mu_2(1-T_i) + \sum_{j=1}^{p+q}\beta_j x_{i,j} + \epsilon_i,
\end{align}
where $\mu_1$ and $\mu_2$ are the main effects of treatments 1 and 2, 
respectively, and $\mu_1-\mu_2$ is the treatment effect. 
Furthermore, $\bm{\beta}=(\beta_1,...,\beta_{p+q})^t$ represents the covariate 
effects, and  $\bm{\epsilon}=(\epsilon_1,...,\epsilon_{n})^t$ is independent and identically distributed 
random 
errors
with mean zero and variance $\sigma_\epsilon^2$, and is independent of 
covariates.
For simplicity, all covariates are assumed to be independent of each other and have 
expectations of zero, i.e., $\mathbb{E}X_j=0$ for $j=1,...,p+q$. 

After the units are allocated into treatment groups via CAR, a working model is used to estimate and test the treatment 
effect.
In such a working model, it is common in practice to include a subset of the
covariates used in randomization or sometimes even no covariates at all 
\citep{Shao2010, Ma2015, Sverdlov2015}.
Therefore, without loss of generality, suppose that the first $p$ covariates 
are included in the working model,
\begin{align}\label{eq:working_model}
\mathbb{E} [Y_i] = \mu_1T_i + \mu_2(1-T_i) + \sum_{j=1}^{p}\beta_j x_{i,j}.
\end{align}
Note that when $q=0$, all of the covariates are included in 
the working model, and when $p=0$, no covariates are included.

Let $\bm{Y}=(Y_1, ..., Y_n)^t$,  
$\bm{T}=(T_1, ..., T_n)^t$,
$\bf{X}=[\bf{X}_{\text{in}}; \bf{X}_{\text{ex}}]$, where
\begin{align*}
\bf{X}_{\text{in}}=\left[\begin{array}{cccccc} 
x_{1,1} & \cdots & x_{1,p}\\
\vdots&\ddots& \vdots\\
x_{n,1} & \cdots & x_{n,p} \\
\end{array}\right], \
\bf{X}_{\text{ex}}=\left[\begin{array}{cccccc} 
x_{1,p+1} & \cdots & x_{1,p+q}\\
\vdots&\ddots& \vdots\\
x_{n,p+1} & \cdots & x_{n,p+q} \\
\end{array}\right].
\end{align*}
Further let $\bm{\beta}_{\text{in}}=(\beta_1,...,\beta_p)^t$,
$\bm{\beta}_{\text{ex}}=(\beta_{p+1},...,\beta_{p+q})^t$, so that
$\bm{\beta}=(\bm{\beta}_{\text{in}}^t,\bm{\beta}_{\text{ex}}^t)^t$.
The working model can then also be written as
$
\mathbb{E} [\bm{Y}]={\bf G}\bm{\theta},
$
where ${\bf G}=[\bm{T}; \bm{1}_n-\bm{T}; \bf{X}_{\text{in}}]$ is the 
design matrix, 
$\bm{\theta}=(\mu_1,\mu_2,\bm{\beta}_{\text{in}}^t)^t$ is the vector of the
parameters of interest, and $\bm{1}_n$ is the $n$-dimensional vector of 
ones. 
Therefore, the ordinary least squares (OLS) estimate of $\bm{\theta}$, $\hat{\bm{\theta}}=(\hat{\mu}_1,\hat{\mu}_2,\hat{\bm{\beta}}_{\text{in}}^t)^t$,  
is
$
\hat{\bm{\theta}}
=({\bf G}^t{\bf G})^{-1}{\bf G}^t\bm{Y}.
$

Under CAR, the treatment assignments $\bm{T}$ depend on 
both $\bf{X}_{\text{in}}$ and  $\bf{X}_{\text{ex}}$.  The distribution of $\hat{\bm{\theta}}$ is often difficult to obtain.
However, testing the treatment effect is often the primary goal when performing a 
comparative study (e.g., a randomized clinical trial).
To detect whether a treatment effect exists, we have the following hypothesis 
testing problem,
\begin{align}\label{eq:test}
H_0:\mu_1-\mu_2=0 \text{ versus }H_1:\mu_1-\mu_2 \neq 0,
\end{align}
with the test statistic
\begin{align*}
S=\frac{\bm{L}^t\hat{\bm{\theta}}}{\sqrt{\hat{\sigma}_{w}^2 \bm{L}^t 
({\bf G}^t {\bf G})^{-1} \bm{L}}},
\end{align*}
where $\bm{L}=(1,-1,0,...,0)^t $ is a vector of length $p+2$,
and $\hat{\sigma}_{w}^2 = \Vert\bm{Y}-{\bf G}\hat{\bm{\theta}}\Vert^2/(n-p-2)$ 
is 
the model-based estimate of the error variance 
$\sigma_{w}^2=\sigma_{\epsilon}^2+\sum_{j=1}^{q}
\beta^2_{p+j} \textup{Var}(X_{p+j})$.
The traditional test rejects 
the null hypothesis at the significance level $\alpha$ if $|S|>z_{1-\alpha/2}$, 
where $z_{1-\alpha/2}$ is $({1-\alpha/2})$-th quantile of a 
standard normal distribution.

In addition to testing the treatment effect, it is often of interest to test 
whether there exist covariate effects. 
A general form of hypothesis testing can be used for any 
linear combinations of the covariate effects.
Let $\bf{C}$ be an 
$m\times(p+2)$ matrix of rank $m$ $(m \le p)$ with entries in the first two columns 
all equal to zero (no treatment effect to test).
Consider the following hypotheses
\begin{align}\label{eq:test_beta}
H_0: {\bf C}\bm{\theta}=\bm{c}_0  \mbox{ versus } H_1: 
{\bf C}\bm{\theta}=\bm{c}_1.
\end{align}
The test statistic is
\begin{align*}
S^* = 
\frac{({\bf C}\hat{\bm{\theta}}-\bm{c}_0)^t
	[{\bf C}({\bf G}^t{\bf G})^{-1}{\bf C}^t]^{-1}
	({\bf C}\hat{\bm{\theta}}-\bm{c}_0)}{m\hat\sigma^2_w}.
\end{align*}

The traditional test rejects the null hypothesis if 
$S^*>\chi^2_{m,(1-\alpha)}/m$, 
where $\chi^2_{m,(1-\alpha)}$ is 
$({1-\alpha})$-th percentile of a $\chi^2$ distribution with $m$ degrees of 
freedom.
We can let ${\bf C}=(0, 0, 1, 0,..., 0)$ to test a single covariate effect of $\beta_1$, and similarly for other covariate effects.

\section{General Properties}\label{sec:HypoTestGeneral}

Based on the framework above, we study the statistical properties of 
estimation and hypothesis testing, i.e., \eqref{eq:test} and 
\eqref{eq:test_beta}, under CAR.
We first introduce two widely satisfied assumptions.

\begin{assumption}\label{assum:prop} Global balance: 
	$n^{-1}\sum_{i=1}^n (2T_i-1)\overset{p}{\rightarrow} 0.$
\end{assumption}

\begin{assumption}\label{assum:diff_means} Covariate balance: 
$n^{-1/2}{\sum_{i=1}^n{\left(2T_i-1\right)\bm{x}_i}}
\overset{d}{\rightarrow}\bm{\xi}
$, where $\bm{\xi}$ is a (p+q)-dimensional random vector with 
$\mathbb{E}[\bm{\xi}]=\bm{0}$.
\end{assumption}

Assumption \ref{assum:prop} requires that the proportions of units in each 
treatment group converge to $1/2$,
which is usually the desired target proportion because balanced treatment 
assignments are more likely to provide efficient estimation and powerful tests.
In contrast, Assumption \ref{assum:diff_means} specifies the asymptotic 
properties of the imbalance vector of covariates, i.e., 
$\sum_{i=1}^{n}(2T_i-1)\bm{x}_i$. That is, the sums of the covariates in each treatment group tend to be 
equal as the sample size increases. Together with Assumption 
\ref{assum:prop}, this implies the similarity of the averages for each covariate 
between the two treatment groups.
These two assumptions ensure that a CAR procedure achieves good balancing properties, both globally and 
across covariates. 

We now present our main theoretical results.

\begin{theorem}
\label{thm:EstConsistent}
Under Assumptions \ref{assum:prop} and \ref{assum:diff_means}, the estimates based on the working 
model 
\eqref{eq:working_model}
are consistent. That is, 
$	\hat{\bm{\theta}} 
\overset{p}{\rightarrow} \bm{\theta}$.
Furthermore,
\begin{align*}
\sqrt{n}(\hat{\bm{\theta}}-\bm{\theta})
=\frac{1}{\sqrt{n}}{\bf V}^{-1}{\bf G}^t\left({\bf 
X}_{\textup{ex}}\bm{\beta}_{\textup{ex}}+\bm{\epsilon}\right)
+\bm{o}_P(1),
\end{align*}
where ${\bf V}=\textup{diag}\left(1/2,1/2,\textup{Var}(X_1),...,
\textup{Var}(X_p)\right)$.
\end{theorem}

The representation provides a convenient means to derive the asymptotic 
distribution of $\hat{\bm{\theta}}$ and its linear combinations.
In particular, for the estimated treatment effect $\hat\mu_1-\hat\mu_2$, 
\begin{align*}
\sqrt{n}[(\hat{\mu}_1-\hat{\mu}_2)-(\mu_1-\mu_2)]=
\frac{2}{\sqrt{n}}\sum_{i=1}^n\left(2T_i-1\right)
\left(\sum_{j=1}^q\beta_{p+j}x_{i,p+j}+\epsilon_i\right)
+o_P(1),
\end{align*}
based on which the asymptotic distribution can be obtained.
We partition $\bm{\xi}=(\bm{\xi} _{\text{in}}^t, \bm{\xi} _{\text{ex}}^t)^t$ 
so that $\bm{\xi} _{\text{in}}$ represents the first $p$ dimensions of 
$\bm{\xi}$
and $\bm{\xi}_{\text{ex}}$ the last $q$ dimensions.
Let $Z$ be a standard normal random variable that is independent of $\bm{\xi} 
_{\text{ex}}$. We have the following corollary.

\begin{corollary}
\label{thm:cor:muDist}
Under the working model \eqref{eq:working_model}, assume that Assumptions 
\ref{assum:prop} and \ref{assum:diff_means} are satisfied. Then
$
\sqrt{n}[(\hat{\mu}_1-\hat{\mu}_2)-(\mu_1-\mu_2)]\overset{d}{\rightarrow} 
2\sigma_{\epsilon} Z + 2\bm{\beta}_{\textup{ex}}^t\bm{\xi} _{\textup{ex}}.
$
\end{corollary}

The corollary describes the asymptotic behavior of $\hat\mu_1-\hat\mu_2$ 
under the working model \eqref{eq:working_model}.
If the model parameters $\bm{\beta}$ and $\sigma_\epsilon^2$ are known, 
statistical inference, such as a Wald-type hypothesis test, can be 
constructed based on the asymptotic distribution.
In practice, these parameters are unknown, and the model-based test procedure 
defined in \eqref{eq:test} is used instead.
It assumes the normal approximation for the asymptotic distribution, and the 
asymptotic variance is estimated by $\hat{\sigma}_{w}^2 \bm{L}^t 
({\bf G}^t {\bf G})^{-1} \bm{L}$, which is shown in the supplementary materials to equal 
$4\sigma_{w}^2/n+o_P(1/n)$.
Furthermore, let $\lambda_1=\sigma_{\epsilon}/\sigma_{w}$ and $\lambda_2=1/\sigma_{w}$.
The asymptotic properties of the test \eqref{eq:test} are presented in the following theorem. 

\begin{theorem}
\label{thm:TestStatDistGeneral}
Under the working model \eqref{eq:working_model},  assume that Assumptions \ref{assum:prop} and 
\ref{assum:diff_means} are satisfied.
\begin{enumerate}
	\item
	When $H_0: \mu_1-\mu_2=0$, then
	$
	S \overset{d}{\rightarrow} \lambda_1 Z + \lambda_2 
	\bm{\beta}_{\textup{ex}}^t\bm{\xi} _{\textup{ex}}.
	$
	
	\item
	When $H_1: \mu_1-\mu_2 \neq 0$, we consider a sequence of local alternatives 
	with $\mu_1-\mu_2=\delta/\sqrt{n}$ for a fixed $\delta \neq 0$, then
	$
	S \overset{d}{\rightarrow} \lambda_1 Z + \lambda_2 
	\bm{\beta}_{\textup{ex}}^t\bm{\xi} 
	_{\textup{ex}}+\lambda_2\delta/2.
	$
\end{enumerate}
\end{theorem}

The asymptotic distribution of test statistic $S$ under $H_0$ consists of two 
independent components, $\lambda_1 Z$ and $\lambda_2 
\bm{\beta}_{\text{ex}}^t\bm{\xi} _{\text{ex}}$.
The first component is due to the random error 
$\epsilon_i$ in the underlying model \eqref{eq:true_model}, and remains 
invariant under various CAR because the 
randomization procedure uses only covariate information and does 
not depend on the observed responses.
In addition, note that $\bm{\xi}_{\text{ex}}$ in the second 
component is the 
last $q$ dimensions of $\bm{\xi}$. 
By Assumption \ref{assum:diff_means}, $\bm{\xi}$ is the asymptotic 
distribution of the imbalance vector of covariates $\sum_{i=1}^{n}(2T_i-1)\bm{x}_i$ and 
illustrates how well the covariates are balanced under a specific 
CAR.
The better it performs in terms of covariate balance, the more concentrated 
$\bm{\xi}$ is distributed around $0$.
Therefore, the second component of $S$ represents the effects of 
CAR on the test statistic through the level of 
covariate balance. 
Depending on the extent to which the covariates are balanced, the test may behave 
differently in terms of size and power.

When the asymptotic distribution of $S$ is no longer a standard normal 
distribution, the traditional test may fail to maintain the pre-specified type 
I error.
Let $s_{1-\alpha/2}$ be $({1-\alpha/2})$-th quantile of the asymptotic 
distribution of $S$ under the null hypothesis.
If $s_{1-\alpha/2}<z_{1-\alpha/2}$, the test is conservative in 
the sense that the actual type I error is smaller than the prespecified level 
$\alpha$.
In fact, such conservativeness is often the case for CAR and can be demonstrated by a comparison of $\bm{\xi}$ between CR and CAR.
Under CR, $\bm{\xi}$ follows a normal distribution that 
makes 
$S$ follow a standard normal distribution asymptotically (Section 
\ref{sec:HypoTestSpecialCR}), in which case the test has valid type 
I error.
However, CAR 
is used to reduce the imbalance 
of covariates between treatment groups, and hence $\bm{\xi}$ is more 
concentrated around $0$ than with CR, leading to 
conservative tests.
Three special cases of CAR (RR, PSR, 
and $D_A$-BCD) are discussed in detail in Sections 
\ref{sec:HypoTestSpecialRR} to
\ref{sec:HypoTestSpecialDABCD}.
The correction of conservative tests is discussed in Section 
\ref{sec:correction}.

In addition to type I error, the explicit form of power can also be derived based 
on Theorem \ref{thm:TestStatDistGeneral}.
Under the local alternatives, $\mu_1-\mu_2=\delta/\sqrt{n}$ for a fixed 
$\delta \neq 0$, the power is
\begin{align*}
\mathbb{P}
(|S|>z_{1-\alpha/2})=F_S(-z_{1-\alpha/2}+\frac{1}{2}\lambda_2\delta)
							+F_S(-z_{1-\alpha/2}-\frac{1}{2}\lambda_2\delta)+o(1),
\end{align*}
where $F_S$ is the cumulative distribution function of the asymptotic 
distribution of $S$.
In Section \ref{sec:Simulation} power is evaluated numerically for several CAR procedures.

In a manner similar to that of the treatment effect, the inference for the 
covariates can also be studied following the representation given in Theorem \ref{thm:EstConsistent}.
The next theorem illustrates the asymptotic normality of 
$\bm{\beta}_{\text{in}}$ under CAR.

\begin{corollary}
	\label{thm:cor:betaDist}
	Under the working model \eqref{eq:working_model},  assume that Assumptions 
	\ref{assum:prop} and \ref{assum:diff_means} are satisfied. Then
	$
	\sqrt{n}(\hat{\bm{\beta}}_{\textup{in}}-\bm{\beta}_{\textup{in}})
	\overset{d}{\rightarrow} 
	N(\bm{0}, \sigma_{w}^2\tilde{\bf V}^{-1}),
	$
	where $\tilde{\bf 
	V}=\textup{diag}(\textup{Var}(X_1),...,\textup{Var}(X_{p}))$.
\end{corollary}

Based on the asymptotic normality of $\hat{\bm{\beta}}_{\text{in}}$, tests for 
these parameters can be constructed with the asymptotic variances 
replaced by their consistent estimates.
The next theorem shows that the traditional test is valid for linear combinations of 
$\bm{\beta}_{\text{in}}$.

\begin{theorem}
	\label{thm:TestStatDistBeta}
Under the working model \eqref{eq:working_model},  assume that Assumptions \ref{assum:prop} and 
\ref{assum:diff_means} are satisfied.
	\begin{enumerate}
		\item
		When $H_0: {\bf C}\bm{\theta}=\bm{c}_0$, then
		$
		S^* \overset{d}{\rightarrow} \chi^2_{m}/m.
		$
		
		\item
		When $H_1: {\bf C}\bm{\theta}=\bm{c}_1$,  we consider a sequence of local 
		alternatives 
		with $\bm{c}_1-\bm{c}_0=\bm{\Delta}/\sqrt{n}$ for a fixed $\bm{\Delta} 
		\neq 
		\bm{0}$, then
		$
		S^* \overset{d}{\rightarrow} \chi^2_{m}(\phi)/m
		$, where $\phi=\bm{\Delta}^t[{\bf C} 
		{\bf V}^{-1}{\bf C}^t]^{-1}\bm{\Delta}/
		\sigma_{w}^2$ is the non-central parameter.
	\end{enumerate}
\end{theorem}

Theorem \ref{thm:TestStatDistBeta} states that the type I error is maintained 
when testing the covariate effects under CAR.
The power, however, is reduced if not all covariate information is 
incorporated in the working model.
Because the inference for covariate effects is valid under 
CAR, the next section focuses mainly on  
testing treatment effect.

\section{Properties of Several CAR procedures}\label{sec:HypoTestSpecial}

\subsection{Complete Randomization}\label{sec:HypoTestSpecialCR}

CR assigns units to each treatment group with the 
equal probability $1/2$.
Because the treatment assignment is independent and does not depend on covariates, it 
follows from the central limit theorem that
$
n^{-1/2}{\sum_{i=1}^{n}{(2T_i-1)\bm{x}_i}}\overset{d}{\rightarrow}
N(\bm{0},\bm{\Sigma})
$
with $\bm{\Sigma}=\textup{diag}(\textup{Var}(X_1),...,\textup{Var}(X_{p+q}))$. 

Therefore, under CR, $\bm{\xi}$ defined in Assumption \ref{assum:diff_means}
is a normal distribution, and furthermore, $\bm{\beta}_{\textup{ex}}^t\bm{\xi} 
_{\textup{ex}} \sim N(0,\sum_{j=1}^{q}\beta^2_{p+j} \textup{Var}(X_{p+j}) )$.
By Theorem \ref{thm:TestStatDistGeneral}, the 
asymptotic null distribution of $S$ is $N(0,1)$.
The traditional test under CR is valid, and no adjustment is needed.

\subsection{Rerandomization}\label{sec:HypoTestSpecialRR}

To balance the covariates across treatment groups, \citet{Morgan2012} propose RR, for which the procedure can be summarized as follows.
A balance criterion needs to be first specified to determine when randomization is acceptable.  
For example, the criterion could be defined as a threshold of $a>0$ on some user-defined imbalance measure, denoted as $M$.
The procedure then randomizes the units into treatment groups using CR until the criterion $M<a$ is satisfied.
The final randomization obtained is used to perform the experiment.

\citet{Morgan2012} choose the imbalance measure to be the 
Mahalanobis distance
\begin{align*}
M & =  (\bm{\bar{x}}_1 - \bm{\bar{x}}_2)^t 
\textup{Cov}(\bm{\bar{x}}_1 - \bm{\bar{x}}_2)^{-1}
(\bm{\bar{x}}_1 - \bm{\bar{x}}_2),
\end{align*}
where $\bm{\bar{x}}_1$ and $\bm{\bar{x}}_2$ are the sample means of the covariates 
in the two treatment groups.
Under the assumption of independent covariates, the Mahalanobis distance can be 
expressed as
\begin{align*}
M=\sum_{j=1}^{p+q}\left(
\frac{\sum_{i=1}^{n} (2T_i-1) x_{i,j}}
{\sqrt{n\textup{Var}(X_j)}}\right)^2+o_P(1).
\end{align*}
By the balance criterion $M<a$,
$
n^{-1/2}\sum_{i=1}^n{(2T_i-1)\bm{x}_i}\overset{d}{\rightarrow}
\bm{\Sigma}^{1/2}
\bm{D} \mid \bm{D}^t\bm{D}<a,
$
where $\bm{\Sigma}^{1/2}$ is the square root of $\bm{\Sigma}$, $\bm{D} \sim 
N(\bm{0},{\bf I}_{p+q})$ 
and ${\bf I}_{p+q}$ is the $(p+q)$-dimensional identity matrix.
Applying Theorem \ref{thm:TestStatDistGeneral}, we have the following results for testing treatment effect under RR.

\begin{theorem}
\label{thm:TestStatDistRR} Under RR, we have
	\begin{enumerate}
		\item
		Under $H_0: \mu_1-\mu_2=0$, then
		$
		S \overset{d}{\rightarrow} \lambda_1 Z + \lambda_2 
		\bm{\beta}_{\textup{ex}}^t\bm{\xi}^{\textup{RR}}_{\textup{ex}},
		$
		where  
		$\bm{\xi}^{\textup{RR}}_{\textup{ex}}$ 
		is the last $q$ dimensions of 
		$\bm{\xi}^{\textup{RR}}=\bm{\Sigma}^{1/2}
		\bm{D} \mid \bm{D}^t\bm{D}<a$.
		\item
		Under $H_1: \mu_1-\mu_2 \neq 0$, where $\mu_1-\mu_2=\delta/\sqrt{n}$ 
		for a fixed $\delta \neq 0$, then
		$
		S \overset{d}{\rightarrow} \lambda_1 Z + \lambda_2 
		\bm{\beta}_{\textup{ex}}^t\bm{\xi}^{\textup{RR}}
		_{\textup{ex}}+\lambda_2\delta/2.
		$
	\end{enumerate}
\end{theorem}

Furthermore, the asymptotic variance of $S$ is 
\begin{align*}
\lambda_1^2+\lambda_2^2\bm{\beta}_{\textup{ex}}^t 
\textup{Var}(\bm{\xi}^{\textup{RR}}_{\textup{ex}}) 
\bm{\beta}_{\textup{ex}}
=\frac{\sigma_{\epsilon}^2+v_a 
	\sum_{j=1}^{q}\beta^2_{p+j}\textup{Var}(X_{p+j})}{\sigma_{\epsilon}^2+\sum_{j=1}^{q}
	\beta^2_{p+j} \textup{Var}(X_{p+j})},
\end{align*}
where $v_a<1$ is defined in \citet{Morgan2012}.
The asymptotic distribution of $S$ under RR is no longer normal, 
and it is more concentrated around 0 than the standard normal 
distribution, which indicates that the traditional test is more 
conservative.
The extent of this conservativeness is affected by the value of $v_a$, 
which is an increasingly monotonic function of $a$. 
By selecting a smaller value of $a$, the covariates are more 
balanced due to stricter balance criterion, resulting in 
a lower asymptotic variance of $S$.
However, a smaller $a$ means that on average it takes more attempts to meet 
the balance criterion.
Further discussion of the choice of $a$ can be found in \citet{Morgan2012}.

\subsection{Pairwise Sequential Randomization}\label{sec:HypoTestSpecialCAM}

Although RR can significantly reduce the covariate balance, it is unable to 
scale up for cases with a large number of covariates or units, 
which are almost ubiquitous in the era of big data.  
PSR, recently proposed by \citet{Qin2016}, solves this problem by sequentially and adaptively assigning units to various 
treatment groups and has shown superior performance in terms of 
covariate balance 
and variance of the estimated treatment effect.

Similar to RR, PSR chooses the Mahalanobis distance as the covariate 
imbalance measure because of its 
affinely invariant property and other desirable properties \citep{Morgan2012}.
The biased coin assignment \citep{Efron1971,Shao2010} is used to minimize the imbalance based on the Mahalanobis distance.
We refer the readers to \citet{Qin2016} for a detailed description of PSR with a comparison with RR.

Under PSR, it is shown in \cite{Qin2016} that
\begin{align}
\label{assum:diff_means_psr}
\sum_{i=1}^n{(2T_i-1)\bm{x}_i}=
\bm{O}_P\left(1\right).
\end{align}
We then have the following theorem.

\begin{theorem}
	\label{thm:TestStatDistCAM} Under PSR, we have
	\begin{enumerate}
		\item
		Under $H_0: \mu_1-\mu_2=0$, then
		\begin{align*}
		S \overset{d}{\rightarrow} 
		N\left(0,\frac{\sigma_{\epsilon}^2}
			{\sigma_{\epsilon}^2+\sum_{j=1}^{q}
			\beta^2_{p+j} \textup{Var}(X_{p+j})}\right).
		\end{align*}
		\item
		Under $H_1: \mu_1-\mu_2 \neq 0$, where $\mu_1-\mu_2=\delta/\sqrt{n}$ 
		for a fixed $\delta \neq 0$, then
		\begin{align*}
			S \overset{d}{\rightarrow} 
				N\left(\frac{1}{2}\lambda_2\delta,\frac{\sigma_{\epsilon}^2}
		{\sigma_{\epsilon}^2+\sum_{j=1}^{q}
			\beta^2_{p+j} \textup{Var}(X_{p+j})}\right).
		\end{align*}
	\end{enumerate}
\end{theorem}

The variance from the covariates is completely eliminated from the numerator 
of the asymptotic distribution of $S$, resulting in a distribution more 
concentrated around 0 than the standard normal distribution. 
In fact, the conclusion under PSR can be extended to a large class of 
CAR if Assumption \eqref{assum:diff_means} is 
replaced by \eqref{assum:diff_means_psr}.
This can be considered as a natural extension of the conditions proposed in 
\citet{Ma2015} that lead to conservative tests for covariate-adaptive designs that
balance discrete covariates. 
Note that condition \eqref{assum:diff_means_psr} is quite strong and is not 
a necessary condition for a conservative test. 
For example, the condition is not satisfied under RR, whereas the test is also 
conservative as shown in Section \ref{sec:HypoTestSpecialRR}.

\subsection{Atkinson's $D_A$-Biased Coin Design}\label{sec:HypoTestSpecialDABCD}

Atkinson's $D_A$-BCD is proposed to 
balance allocations across covariates to minimize the variance of 
estimated treatment effects when a classical linear model is assumed
between a response and the covariates \citep{Atkinson1982, Smith1984, 
Smith1984b}.
Unlike RR, it is used in settings in which covariate information 
is collected sequentially, such as in clinical trials.

$D_A$-BCD sequentially assigns units to treatment groups with an adaptive 
allocation probability:
suppose $n$ units have been assigned to treatment groups, 
$D_A$-BCD assigns the $(n+1)$-th unit to treatment 1 with 
probability 
\begin{align*}
\frac{[1-(1;\bm{x}_{n+1}^t)({\bf F}_n^t{\bf F}_n)^{-1}\bm{b}_n]^2}
{[1-(1;\bm{x}_{n+1}^t)({\bf F}_n^t{\bf F}_n)^{-1}\bm{b}_n]^2+
	[1+(1;\bm{x}_{n+1}^t)({\bf F}_n^t{\bf F}_n)^{-1}\bm{b}_n]^2},
\end{align*}
where ${\bf F}_n=[\bm{1}_n;{\bf X}]$ and 
$\bm{b}_n^t=(2\bm{T}-\bm{1}_n)^t{\bf F}_n$.

Applying result (10.5) of \citet{Smith1984b}, we obtain that
$
n^{-1/2}\sum_{i=1}^{n}{(2T_i-1)\bm{x}_i}\overset{d}{\rightarrow}
N(\bm{0},\bm{\Sigma}/5).
$
It is clear to see that under $D_A$-BCD the variance of $\sum_{i=1}^{n}(2T_i-1)\bm{x}_i$ is reduced to $1/5$ of that under 
CR, which indicates that covariates are more balanced than under CR. 
The next theorem states the asymptotic distributions of $S$ under $D_A$-BCD.

\begin{theorem}
	\label{thm:TestStatDistDABCD} Under $D_A$-BCD, we have
	\begin{enumerate}
		\item
		Under $H_0: \mu_1-\mu_2=0$, then
		\begin{align*}
		S \overset{d}{\rightarrow} 
		N\left(0,\frac{\sigma_{\epsilon}^2+\frac{1}{5} 
			\sum_{j=1}^{q}\beta^2_{p+j}\textup{Var}(X_{p+j})}
		{\sigma_{\epsilon}^2+\sum_{j=1}^{q}
			\beta^2_{p+j} \textup{Var}(X_{p+j})}\right).
		\end{align*}
		\item
		Under $H_1: \mu_1-\mu_2 \neq 0$, where $\mu_1-\mu_2=\delta/\sqrt{n}$ 
		for a fixed $\delta \neq 0$, then
		\begin{align*}
		S \overset{d}{\rightarrow} 
		N\left(\frac{1}{2}\lambda_2\delta,\frac{\sigma_{\epsilon}^2+\frac{1}{5} 
			\sum_{j=1}^{q}\beta^2_{p+j}\textup{Var}(X_{p+j})}
		{\sigma_{\epsilon}^2+\sum_{j=1}^{q}
			\beta^2_{p+j} \textup{Var}(X_{p+j})}\right).
		\end{align*}
	\end{enumerate}
\end{theorem}

This theorem shows that the test statistic $S$ has asymptotic variance of less than 1, so the test is conservative with a reduced type I error for 
testing the treatment effect. 

Based on the above four randomization procedures, our findings are as follows.
(i) Under CR, the distribution
of $S$ remains asymptotically standard normal, so it provides the correct type I error because CR does not use covariate information 
at the assignment stage. 
(ii) Under the other three procedures (RR, $D_A$-BCD, and PSR), the asymptotic distributions of $S$ are not standard
normal; therefore, their type I errors (based on $S$) are no longer correct. 
In the next section, we discuss the correction of type I error for CAR 
procedures, and we may then compare their adjusted powers. 

We may also apply the general theorems 
(in Section \ref{sec:HypoTestGeneral}) to other CAR procedures. 
In particular, stratified block randomization and Pocock and Simon's minimization method are two of the most popular CAR designs.
To deal with continuous covariates, discretization is required to implement these methods.
Let $d_j(X_j)$ be the discretized version of $X_j$, $j=1,...,p+q$,  that is used in randomization.
\cite{Ma2015} showed that Assumptions \ref{assum:prop} and \ref{assum:diff_means} are satisfied for both methods where $\bm{\xi}$ is a $(p+q)$-dimensional normal distribution with mean zero and a diagonal covariance matrix with the $j$-th element being $\mathbb{E}[\textup{Var}(X_j \mid d(X_j))]$.

\section{Correction for Conservativeness}\label{sec:correction}

It can be seen that most of the CAR procedures lead to 
conservative type I error when testing the treatment effect because the traditional 
tests use a standard normal distribution as the null distribution.  
Therefore, we propose the following approach to correct the conservative type I 
errors and to obtain higher powers.

Because we derived the asymptotic distribution of $S$ in 
Theorem \ref{thm:TestStatDistGeneral}, 
we can obtain the correct asymptotic critical values.
However, because the asymptotic distribution depends on unknown parameters, they must be estimated using the observed sample to obtain the approximated null 
distribution and adjust the corresponding critical values and p-values.

After adjusting the critical values and p-values, more 
powerful hypothesis testing results can be obtained. The more conservative the traditional 
tests are, the more powerful their corrected versions become.  
Finally, we compare the CAR procedures mentioned 
above in terms of covariate balance, conservativeness of the traditional tests, 
and powers of the corrected test, and their advantages and 
disadvantages are summarized in Table \ref{tab:summary}.  
The conclusions in the table are further verified via simulation in the next section.

An alternative approach to preserving the type I error rate is the randomization test \citep{Simon1979}.
By fixing the covariate values and responses, it simulates the treatment assignments with the corresponding CAR procedure.
The approximate null distribution of the test statistic can then be generated, upon which the p-value is determined.
Unlike the model-based tests, the randomization test requires no distributional assumptions, although it is more computationally intensive and may lead to a loss of power.
This topic is discussed further in \cite{Rosenberger2015} and the references therein.

\addtolength{\tabcolsep}{-6.5pt} 
\begin{table}
\begin{center}
\begin{tabular}{c c c c}
                 Randomization & Covariate balance & Type I error of traditional
                 test & Power of corrected test\\
  \hline
                CR      & least balanced & valid  & least powerful \\
                RR     & moderately balanced & moderately conservative & 
                moderately 
                powerful \\
                $D_A$-BCD      & moderately balanced & moderately conservative 
                & 
                moderately powerful \\
                PSR      & most balanced & most conservative & most powerful \\
   \hline
\end{tabular}
\caption{Comparison of covariate-adaptive randomization procedures in 
terms of the covariate 
balance, the traditional tests' conservativeness, and the corrected tests' powers.}
\label{tab:summary}
\end{center}
\end{table}
\addtolength{\tabcolsep}{6.5pt}

\section{Numerical Studies} \label{sec:Simulation}
\subsection{Verification of Theoretical Results}

We first verify the theoretical asymptotic distribution of $S$ under CR, RR, 
PSR, and $D_A$-BCD.  
Assume the underlying model is 
$Y_{i}=\mu_1T_i+\mu_2(1-T_i)+\sum_{j=1}^{4}\beta_j 
x_{i,j}+\epsilon_{i}$ where $\mu_1=\mu_2=0$, $\beta_j=1$ for $j=1,...,4$. 
$x_{i,j} \sim N(0, 1)$ for $j=1,...,4$, 
and the covariates are independent of each other.  
The random error $\epsilon_{i} \sim N(0, 2^2)$ is independent 
of all $x_{i,j}$.
We simulate the data according to the underlying model with a sample size $n=500$ 
and use the working model that includes only two of the four covariates, 
$\mathbb{E} [Y_i]=\mu_1T_i+\mu_2(1-T_i)+\beta_1 x_{i,1}+\beta_2 x_{i,2}$, 
to obtain the test statistic $S$.
In Figure \ref{fig:compare_dist}, we plot the simulated distributions of $S$ 
along with the theoretical distributions of $S$ given by Theorem 
\ref{thm:TestStatDistGeneral}.
As the figure shows, the theoretical distributions are very close to the 
simulated distributions for all randomization procedures, which verifies our 
theoretical results.
For comparison, the standard normal distribution is plotted in bold gray.
As we move from the left panel to the right panel (i.e., from CR 
to PSR), the distribution of $S$ becomes narrower.  
Therefore, the use of the critical values or p-values obtained from the standard 
normal distribution will result in conservative tests with reduced type I 
errors for RR, PSR, and $D_A$-BCD.  
The correction for such conservativeness is further illustrated in the 
following sections.
\begin{figure} \centering
	\includegraphics[scale=0.55]{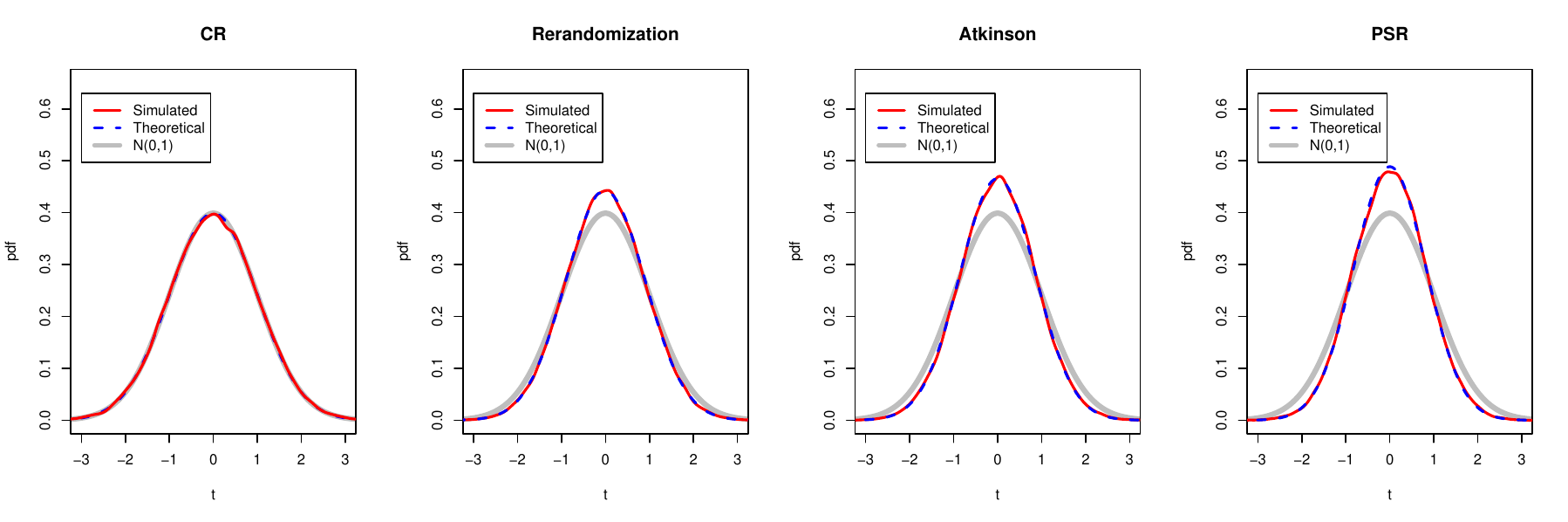}
	\caption{Comparison of theoretical distributions and simulated 
	distributions of $S$.  
	From left to right: CR, RR, $D_A$-BCD, and PSR.  
	In each panel, red solid curve represents the simulated distribution, 
	blue dashed curve represents the theoretical distribution, 
	and bold gray curve is the standard normal density.} 
\label{fig:compare_dist}
\end{figure}

\subsection{Conservative Hypothesis Testing for Treatment Effect}\label{sec:SimulationConservativeTest}

The previous sections showed that the traditional test for treatment 
effect under most CAR procedures generates 
conservative results.  
In this section, we verify this phenomenon. 
Suppose the underlying model is
\begin{align}\label{eq:simu_true_model}
Y_{i}=\mu_1T_i+\mu_2(1-T_i)+\sum_{j=1}^{6}\beta_j x_{i,j}+\epsilon_{i},
\end{align}
where $\beta_j=1$ for $j=1,...6$.
$x_{i,j} \sim N(0, 1)$, and the covariates are independent of each other. 
The random error $\epsilon_{i} \sim N(0, 2^2)$ is 
independent of all $x_{i,j}$.
We use the following four working models to test the treatment effect, i.e., 
$H_0: \mu_1=\mu_2$ and $H_1: \mu_1 \neq \mu_2$.

W1: $\mathbb{E} [Y_i]=\mu_1T_i+\mu_2(1-T_i)$.

W2: $\mathbb{E} [Y_i]=\mu_1T_i+\mu_2(1-T_i)+\sum_{j=1}^{2}\beta_j 
x_{i,j}$.

W3: $\mathbb{E} [Y_i]=\mu_1T_i+\mu_2(1-T_i)+\sum_{j=3}^{6}\beta_j 
x_{i,j}$.

W4: $\mathbb{E} [Y_i]=\mu_1T_i+\mu_2(1-T_i)+\sum_{j=1}^{6}\beta_j 
x_{i,j}$.

Note that the first working model is equivalent to the two-sample t-test, and the last 
working model is the same as the underlying model.
We simulate data according to \eqref{eq:simu_true_model} with 
$\mu_1-\mu_2=0$ and $\mu_1-\mu_2=0.3$, respectively, to obtain the type I errors and powers of the traditional tests; the sample size is $n=500$.
The results are shown in the column ``Traditional test'' in Table \ref{tab:trt_effect}.
Under CR, all working models provide correct type I 
errors.
However, under RR, $D_A$-BCD, and PSR, W1, W2, and W3 generate conservative type I errors below 5\%, with PSR being the most conservative.
These results show that CAR leads to conservative 
results for traditional tests of treatment effect.
The more balanced covariates provided by the randomization 
procedures, 
the more conservative the tests become.
As the type I errors are conservative, the powers of RR, 
$D_A$-BCD, and PSR are also affected.

\addtolength{\tabcolsep}{-1pt} 
\begin{table}
	\begin{center}
		\begin{tabular}{cccccccccc}
			\hline
			&  & \multicolumn{4}{c}{Traditional test} &  \multicolumn{4}{c}{Corrected test} \\
			$\mu_1-\mu_2$&	Randomization & W1 & W2 & W3 & W4 & W1 & W2 & W3 & W4 \\
			\hline
			0&	CR       			    & 0.053 & 0.051 & 0.054 & 0.051
											 & 0.048 & 0.050 & 0.046 & 0.045\\
			&	RR          			 & 0.011 & 0.017 & 0.026 & 0.050
										     & 0.051 & 0.050 & 0.052 & 0.051\\
			&	$D_A$-BCD   	  & 0.007 & 0.012 & 0.025 & 0.053
										     & 0.051 & 0.052 & 0.053 & 0.051\\
			&	PSR                     & 0.002 & 0.006 & 0.018 & 0.052
										     & 0.060 & 0.058 & 0.050 & 0.048 \\
			\hline
		0.3&	CR       			    & 0.179 & 0.215 & 0.273 & 0.387
											 & 0.194 & 0.238 & 0.272 & 0.414\\
			&	RR          			 & 0.126 & 0.171 & 0.248 & 0.384
											 & 0.285 & 0.289 & 0.357 & 0.387\\
			&	$D_A$-BCD   	   & 0.112 & 0.155 & 0.244 & 0.386
										     & 0.317 & 0.337 & 0.365 & 0.386\\
			&	PSR                     & 0.087 & 0.140 & 0.235 & 0.386
											 & 0.421 & 0.413 & 0.408 & 0.394 \\
			\hline
		\end{tabular}
		\caption{Type I error and power ($\mu_1-\mu_2=0.3$) of traditional test and corrected test for treatment effect under various working models and randomization procedures.}
		\label{tab:trt_effect}
	\end{center}
\end{table}
\addtolength{\tabcolsep}{1pt}

\subsection{Corrected Hypothesis Testing for Treatment Effect}\label{sec:SimulationCorrectedTest}

In practice, to correct the type I error and obtain higher powers,
we can estimate the critical values and p-values based on the estimated 
asymptotic distribution (Section \ref{sec:correction}).  
Using this approach, we repeat the simulations in the previous 
section; the type I errors and the powers for $\mu_1-\mu_2=0.3$ are presented in the column ``Corrected test'' in Table \ref{tab:trt_effect}.
We also simulate the powers for various $\mu_1-\mu_2$ with values in the range from 0 to 1 and plot the results in Figure \ref{fig:power_asym_n500}.
Additional simulation results with a sample size $n=200$ are provided in the supplementary materials.

\begin{figure} \centering
\includegraphics[scale=0.55]{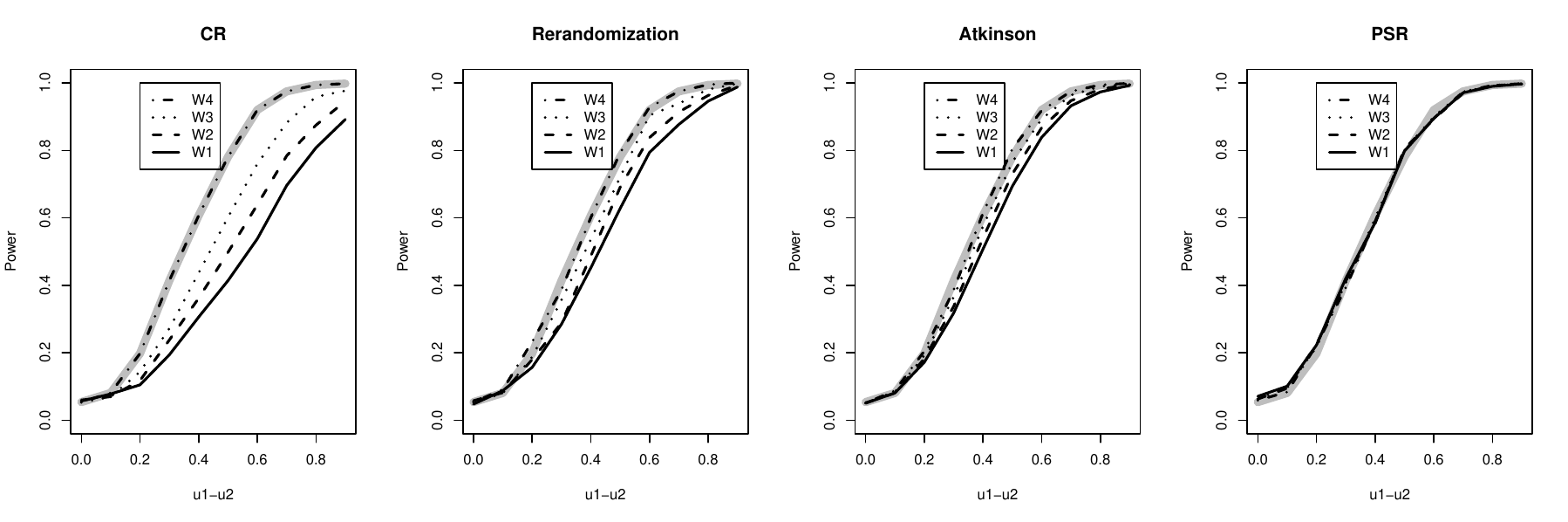}
\caption{Power against $\mu_1-\mu_2$ using 
estimated asymptotic distribution's 
critical values and p-values.  
Sample size $n=500$.  
From left to right: 
CR, 
RR with $a=3$, 
$D_A$-BCD, 
and PSR.
Note that the power of W4 under CR is plotted in bold gray curves in all panels to better allow comparison among randomizations.
}
\label{fig:power_asym_n500}
\end{figure}

Table \ref{tab:trt_effect} shows that all type I errors of the corrected tests are 
successfully controlled at 5\%, which means that the proposed approach works well.
In Figure \ref{fig:power_asym_n500}, as $\mu_1-\mu_2$ increases 
away from 0, the 
powers generally increase.  
However, under CR, different working models provide different powers. 
The more covariates included in the working model, the higher the power.  
CR cannot balance the covariate well, and the covariates not 
included in the working model will affect the test for treatment effect.  
In contrast, under PSR, all working models provide similar powers because 
PSR can balance all covariates well.  
Because RR and $D_A$-BCD can also balance the covariates, but not as well as PSR, their powers are slightly better than that of CR but much worse than that of PSR.

\subsection{Hypothesis Testing for Covariate Effect}

\begin{table}
\begin{center}
\begin{tabular}{c c c}
  \hline
                   Randomization & W3 & W4 \\

  \hline
                CR      & 0.053 & 0.053\\
                RR      & 0.052 & 0.051\\
                $D_A$-BCD   & 0.049 & 0.043\\
                PSR     & 0.051 & 0.050\\
   \hline
\end{tabular}
\caption{Type I error of hypothesis testing for covariate effect $H_0: 
\beta_3=0$ using 
unadjusted critical values under various working models and randomization procedures.}
\label{tab:cov_effect_size}
\end{center}
\end{table}

Finally, we compare the performance of the traditional test for the third 
covariate effect, i.e., $H_0: \beta_3=0$ and $H_1: 
\beta_3 \neq 0$.  
We adopt the same setting from the previous section and choose a range of 
values from 0 to 1 for $\beta_3$ to calculate the power under various working 
models.
Note that in this case, only W3 and W4 contain the third covariate.
The type I errors are shown in Table \ref{tab:cov_effect_size}, and the powers are shown in 
Figures \ref{fig:power_cov_n500}.
The type I errors are all controlled at 5\%, which is consistent 
with our theoretical results in Theorem 
\ref{thm:TestStatDistBeta}.
In other words, no correction is needed to test the covariate effect.
In contrast, Figure \ref{fig:power_cov_n500} shows that the powers are reduced if the working model does not include all covariates.
This is again consistent with the results in Theorem 
\ref{thm:TestStatDistBeta}.  
It is worthwhile to note that the performance of the hypothesis testing for the 
covariate effect does not depend on the choice of randomization procedure, as 
all panels of Figures \ref{fig:power_cov_n500} are nearly identical.

\begin{figure} \centering
\includegraphics[scale=0.55]{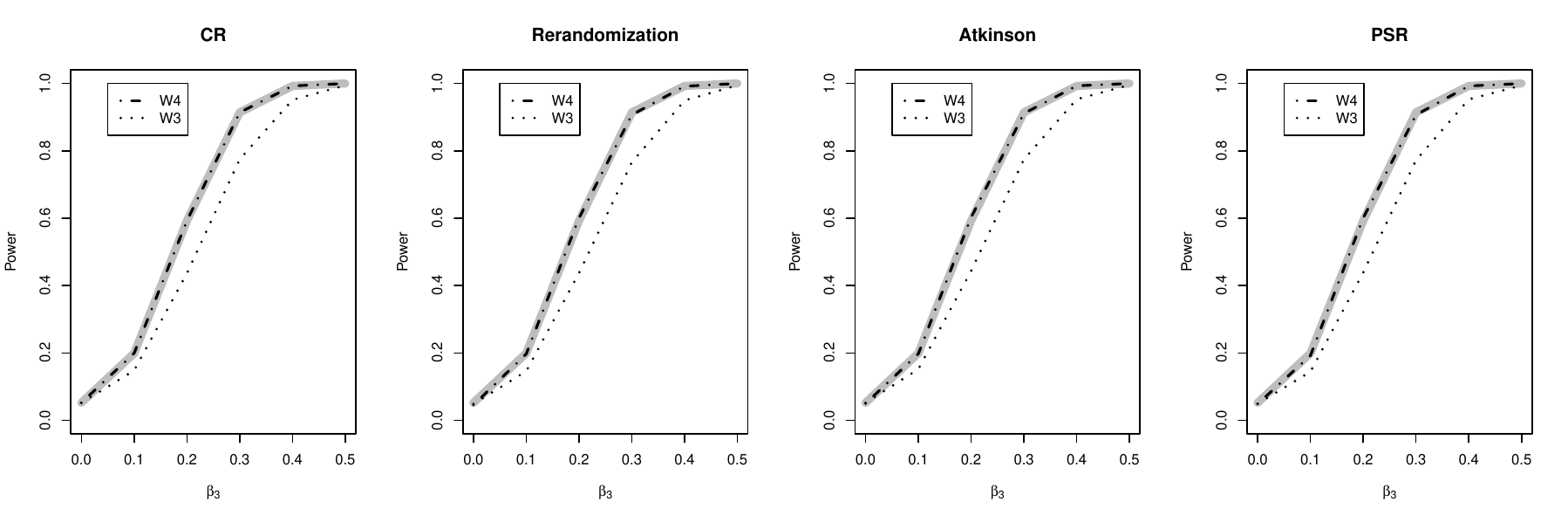}
\caption{Power against $\beta_3$.  
Sample size $n=500$.  
From left to right: 
CR, 
RR with $a=3$, 
$D_A$-BCD, 
and PSR.
Note that the power of W4 under CR is plotted in bold gray curves in all panels to better allow comparison among randomizations.
}
\label{fig:power_cov_n500}
\end{figure}

\section{Real Data Example}
\label{sec:RealData}

We present a real data example to illustrate the use of the proposed method.  The example is based on a randomized, double-blind, parallel-group comparison clinical trial study on the effects of gliclazide and pioglitazone on metabolic control in drug-na\"ive patients with type 2 diabetes mellitus \citep{Charbonnel2005}. The outcome variable is the decrease in HbA1C level after 52 weeks of treatment. The patients' covariates include BMI, age, baseline fasting insulin level, and baseline fasting glucose level, all of which are continuous.  We use some summary statistics of the majority of patients in the study to create a synthetic dataset of patients by simulating their covariates and outcomes with a sample size of $n=850$ via the following fitted linear regression
\begin{align*}
\textup{HbA1C} \ =&\quad 0.084*T-0.039*(1-T) + 0.015*\textup{BMI} + 
0.001*\textup{Age} \\
&-0.001*\textup{BaselineFastingInsulin} + 
0.094*\textup{BaselineFastingGlucose} + \epsilon.
\end{align*}
We assume $\epsilon \sim N(0,0.6215^2)$ in this example. Here, the treatment effect of pioglitazone ($T=1$) over gliclazide ($T=0$) is 0.123.

For comparison, we reassign these patients to the two treatment groups with various randomization procedures, including CR, PS, $D_A$-BCD, and PSR.
We also include stratified block randomization (SB) and Pocock and Simon's minimization method (PS), both of which can be applied to discretized covariates.
After the assignment, we simulate outcome variables for each patient with the fitted regression.  We further conduct hypothesis testing for the treatment effect using the working model with all covariates and the model with no covariates.  The tests include the traditional test, the corrected test, and the randomization test.
Each test is simulated 10,000 times.
The results are presented in Table \ref{tab:power_comp}.

\begin{table}
\begin{center}
	\begin{tabular}{ccccccc}
		\hline
		& \multicolumn{2}{c}{Traditional test} & \multicolumn{2}{c}{Corrected test} & \multicolumn{2}{c}{Randomization test} \\
		Randomization & All cov. & No cov. &  All cov. & No cov. & All cov. & No cov. \\
		\hline
		CR          & 0.821 & 0.743 & 0.821 & 0.743 & 0.816 & 0.742 \\
		SB         & 0.819 & 0.759 & 0.819 & 0.796 & 0.816 & 0.791 \\
		PS          & 0.820 & 0.761 & 0.821 & 0.797 & 0.820 & 0.793 \\
		RR          & 0.817 & 0.766 & 0.816 & 0.814 & 0.815 & 0.791 \\
		$D_A$-BCD   & 0.818 & 0.761 & 0.818 & 0.807 & 0.816 & 0.803 \\
		PSR         & 0.858 & 0.803 & 0.858 & 0.858 & 0.853 & 0.852 \\
		\hline
	\end{tabular}
	\caption{Comparison of power of traditional test, corrected test, and randomization test under various working models (i.e., with all covariates and with no covariates) and randomization procedures.}
	\label{tab:power_comp}
\end{center}
\end{table}

The table shows that with the traditional test and the working model that includes all covariates, the powers under various randomization procedures are all very high.  However, if the working model does not include the covariates used in randomization, the powers are degraded.  
PSR shows the least degradation, $D_A$-BCD and RR show moderate degradation, and CR shows the most degradation.  
The degradations are also significant for SB and PS. This evidence shows the importance of adjusting the hypothesis testing procedure.

In contrast, the degradation is less significant using the corrected test with the proposed method.
In fact, PSR's corrected power under the working model with no covariates is the same as that of the working model with all covariates, which is consistent with the theoretical results.  $D_A$-BCD and RR also benefit significantly from the corrected tests.  Note that our proposed method also works with SB and PS; thus, we can also provide correction for them and obtain clear improvement.  

Finally, the power of the randomization test is very similar and slightly lower than that of the corrected test but much higher than that of the traditional test under various CAR procedures when no covariates are used in the analysis. 
The results also show that better covariate balance improves the power of the randomization test.
In particular, under PSR, the randomization test has the highest power and appears equivalent to the model in which all covariates are adjusted. 
Note that the validity of the corrected test depends on the correctness of model assumptions (e.g., linearly additive covariate effects, no time trend), while the randomization test is almost assumption-free. Thus, the randomization test following CAR is a useful alternative to the model-based methods.

\section{Extension to Logistic Regression Models}
\label{sec:Logistic}

The framework can be extended to more general models than linear regression.
Here, we briefly explore the extension to logistic regression models.

Consider the same setting and notation as in Section \ref{sec:Framework}, except for the outcome $Y_i$ that is a binary variable with possible values 0 and 1.
A logistic regression model relating $Y_i$ to treatment assignment $T_i$ and covariate $\bm{x}_i$ is given by
\begin{align}\label{eq:true_model_log}
\mathbb{E} [Y_i \mid T_i,\bm{x}_i]&= h\Big(\mu_1T_i + \mu_2(1-T_i) + \bm{\beta}^t\bm{x}_i \Big),
\end{align}
with $h(t)=\exp(t)/[1+\exp(t)]$.

To test the treatment effect 
$
H_0:\mu_1-\mu_2=0 \text{ versus }H_1:\mu_1-\mu_2 \neq 0,
$
it may be possible to carry out logistic regression by including all covariates, which is valid if the model \eqref{eq:true_model_log} is correctly specified \citep{Shao2010}.
As discussed in Section \ref{sec:Introduction}, however, the covariates are commonly excluded from analysis under CAR.
For illustration, we consider the special case of a working model with no adjustment for covariates,
\begin{align}\label{eq:working_model_log}
\mathbb{E} [Y_i \mid T_i]=h\Big(\mu_1T_i + \mu_2(1-T_i) \Big).
\end{align}
The cases with partial covariates are evaluated via simulation.
Under the working model \eqref{eq:working_model_log}, the maximum likelihood estimators (MLE) of $\mu_1$ and $\mu_2$ are given by
\begin{align*}
\hat{\mu}_1=h^{-1}\left(\frac{\sum_{i=1}^{n} T_iY_i}{n_1}\right), \ 
\hat{\mu}_2=h^{-1}\left(\frac{\sum_{i=1}^{n} (1-T_i)Y_i}{n_2}\right),
\end{align*}
where $n_1=\sum_{i=1}^{n}T_i$ and $n_2=\sum_{i=1}^{n}(1-T_i)$.
The usual Wald test statistic is then
\begin{align*}
S_L=\frac{\hat{\mu}_1-\hat{\mu}_2}{\sqrt{\widehat{\textup{Var}}(\hat{\mu}_1-\hat{\mu}_2)}},
\end{align*}
where $\widehat{\textup{Var}}(\hat{\mu}_1-\hat{\mu}_2)$ denotes the estimated variance of $(\hat{\mu}_1-\hat{\mu}_2)$ derived from the inverse of the estimated Fisher information matrix.
The null hypothesis is rejected at the significance level $\alpha$ if $|S_L|>z_{1-\alpha/2}$.

The primary goal here is to investigate the type I error rate under the working model \eqref{eq:working_model_log}.
To do so, we introduce another assumption on the covariate balance.
\begin{assumption}\label{assum:diff_log} Covariate balance under logistic regression: 
	$$
	n^{-1/2}{\sum_{i=1}^n{\left(2T_i-1\right)	\Big(h(\mu+\bm{\beta}^t\bm{x}_i)-\mathbb{E}h(\mu+\bm{\beta}^t\bm{x}_i)\Big)}}
	\overset{d}{\rightarrow}{\eta},
	$$
	where ${\eta}$ is a random variable with 
	$\mathbb{E}[{\eta}]={0}$.
\end{assumption}

Compared to Assumption \ref{assum:diff_means}, a nonlinear form of the covariates $h(\mu+\bm{\beta}^t\bm{x}_i)$, instead of $\bm{x}_i$ itself, is assumed to be balanced.
The assumption is equivalent to Assumption \ref{assum:diff_means} if $h$ is the identity function, so it can be considered as an extension of Assumption \ref{assum:diff_means}.
We now give the main results under logistic regression.

\begin{theorem}
	\label{thm:TestStatLogistic}
	Under the working model \eqref{eq:working_model_log}, assume that Assumptions \ref{assum:prop} and 
	\ref{assum:diff_log} are satisfied. When $H_0: \mu_1=\mu_2=\mu$, then
	$
	S_L \overset{d}{\rightarrow} \lambda_1^{'} Z + \lambda_2^{'} {\eta},
	$
	where $Z$ is a standard normal random variable that is independent of ${\eta}$, $\lambda_1^{'}=\sqrt{\mathbb{E}[\textup{Var}(Y_i \mid \bm{x}_i)]/\textup{Var}(Y_i)}$, and $\lambda_2^{'} =\sqrt{1/\textup{Var}(Y_i)}$.
\end{theorem}

Similar to Theorem \ref{thm:TestStatDistGeneral}, the asymptotic distribution of $S_L$ under $H_0$ is the sum of two independent random variables, and the second component depends upon which specific randomization procedure is used.
Under CR, it is easy to see that $\eta \sim N(0, \textup{Var}[\mathbb{E}(Y_i \mid 
\bm{x}_i)])$, so $S_L$ is asymptotically standard normal and the test is valid.
Under CAR, the explicit form of $\eta$ is usually unknown, but it is anticipated that $\eta$ is more concentrated around zero than CR, which makes the test conservative.

To numerically evaluate type I errors under various randomization procedures, we conduct simulations as in Section \ref{sec:SimulationConservativeTest} by assuming a logistic regression model with $\mu_1=\mu_2=0$.
The results are listed in Table \ref{tab:conservative_size_log} and show patterns similar to those in the linear models.
The type I errors are valid for all working models under CR, whereas the tests are conservative under CAR if some covariates are omitted.

In contrast to the linear models, the unknown form of $\eta$ poses difficulties in the correction of type I errors using the theoretical distribution given in Theorem \ref{thm:TestStatLogistic}.
The properties of $\eta$, which describes the asymptotic behavior of the nonlinear quantity specified in Assumption \ref{assum:diff_log}, are usually unclear, in part because many CAR procedures are developed under the linear model framework, such as Atkinson's $D_A$-BCD and PSR.
It is thus desirable that a randomization procedure can produce a known form of $\eta$ with minimum variance, but this is beyond the scope of our discussion.
It is also difficult to derive the asymptotic distribution of $S_L$ under the alternative hypothesis.
These problems are left for future studies.

\begin{table}
	\begin{center}
		\begin{tabular}{c c c c c}
			\hline
			Randomization & W1 & W2 & W3 & W4 \\
			
			\hline
			CR      & 0.050 & 0.051 & 0.049 & 0.050\\
			RR      & 0.021 & 0.025 & 0.033 & 0.047\\
			$D_A$-BCD   & 0.014 & 0.022 & 0.030 & 0.049\\
			PSR     & 0.012 & 0.018 & 0.030 & 0.050\\
			\hline
		\end{tabular}
		\caption{Type I error of traditional test for treatment effect under 
			various working models and randomization procedures.}
		\label{tab:conservative_size_log}
	\end{center}
\end{table}

\section{Conclusion}\label{sec:Conclusion}

In this article, the properties of statistical inference are investigated for general CAR procedures.
The properties are illustrated and applied to several examples, including RR, PSR, and $D_A$-BCD.
A new test is also proposed that can effectively maintain the type I error and improve the power under CAR.
This work provides a theoretical foundation for the analysis of covariate-adaptive randomized experiments.

The linear models with continuous outcomes are assumed in our proposed framework.
However, in practice, it is common to see other types of outcomes.
In Section \ref{sec:Logistic}, we briefly explore the extension to the logistic regression models under which the traditional Wald test is studied.
It is desirable to extend the framework in this article to other nonlinear models, such as generalized linear models and survival analyses \citep{Shao2013, Luo2016, Xu2016}.
It is also of interest to generalize the results to the cases of unequal allocation or multiple treatments \citep{Tymofyeyev2007}.

We focus mainly on covariate-adaptive designs that aim to achieve balance on the covariates.
Such balanced designs are commonly used in practice because it is more convincing to attribute any observed effect to the treatment being tested.
Other covariate-adaptive designs, especially some Atkinson-type designs, are also proposed to increase efficiency.
Under homogeneous linear models, it has been shown that a balanced design is associated with efficiency or high power.
However, covariate balance is not equivalent to efficiency under heteroscedastic nonlinear models and may even reduce the efficiency in some cases \citep{Rosenberger2015}. 
Efficiency should thus be considered when extending the proposed framework to nonlinear models with unbalanced designs. 

\section*{Supplementary Materials}
The supplementary materials contain the proofs of the main theorems, additional simulation results, and the R code.

\bibliographystyle{apalike}
\bibliography{YichenBib}

\end{document}